\documentclass[11pt]{amsart}
\usepackage[text={5in,8in},centering]{geometry}               
\usepackage{graphicx}
\usepackage{amssymb}
\usepackage{epstopdf}
\usepackage{setspace}

\title[Hereditary Recurrent Configurations and Spanning Trees]{A Bijection Between the 
Recurrent Configurations of a
         Hereditary Chip-Firing Model 
         and Spanning Trees}
         
\author{Spencer Backman}

\newtheorem{theorem}{Theorem}
\newtheorem{lemma}{Lemma}

\begin{document}

\maketitle

\begin{abstract}
Hereditary chip-firing models generalize the Abelian sandpile model and the cluster firing model to an exponential family of games induced by covers of the vertex set.  This generalization retains some desirable properties, e.g. stabilization is independent of firings chosen and each chip-firing equivalence class contains a unique recurrent configuration.  In this paper we present an explicit bijection between the recurrent configurations of a hereditary chip-firing model on a graph and its spanning trees.

 \

\


\end{abstract}

\onehalfspacing

\

\begin{section}{Introduction}

Chip-firing on graphs has been studied by several different communities over the past 25 years.  In statistical physics it was introduced by Dhar \cite{Dhar} as an example of self organized criticality as proposed in the Bak-Tang-Wiesenfeld model \cite{Bak}.  In graph theory it was investigated by Tardos \cite{Tardos} and  Bj\"orner, Lovasz and Shur \cite{Bjorner} extending ideas introduced by Spencer \cite{Spencer}.  Later Biggs   \cite{Biggs} studied this game as related to algebraic potential theory on graphs.  Most recently, Baker and Norine \cite{Baker} have shown that using the language of chip-firing, one can derive a Riemann-Roch theorem for graphs analogous to the classical statement from algebraic geometry. 

In the Abelian sandpile model, vertices are restricted to fire individually.  This is in contrast to the cluster firing model where vertices are allowed to fire simultaneously.  A chip-firing model is a collection $H$ of subsets of the vertex set, those subsets which are allowed to fire simultaneously if no vertex is sent into debt.  If every vertex appears somewhere in $H$, and the family $H$ has the hereditary property, we say that $H$ is a {\it hereditary chip-firing model}.  From this perspective, the sandpile model is the coarsest hereditary chip-firing model, described by taking $H$ to be the collection of all singleton sets from $V(G)\setminus \{ v_0\}$, and the cluster model is the finest hereditary chip-firing model, described by taking $H$ to be the power set of $V(G)\setminus \{ v_0\}$. 

Some of the fundamental properties of the Abelian sandpile model and the cluster firing model extend to arbitrary hereditary chip-firing models: the stabilization of a configuration is independent of the firings chosen and each chip-firing equivalence class contains a unique recurrent configuration.  It is well known that the number of chip-firing equivalence classes is the same as the number of spanning trees of a graph.  It follows that the number of recurrent configurations in a hereditary chip-firing models is the same as the number of spanning trees.  

For the case of ASM and CFM, several bijections between recurrent configurations and spanning trees exist in the literature, e.g. \cite{Dhar} \cite{Cori} \cite{Biggs2} \cite{Chebkin}.  There is a simple relationship between the recurrent configurations in ASM and CFM which allows a bijection is one model to be ``dualized" to produce a bijection in the other model.  The recurrent configurations in CFM go by several names: $G$-parking functions, $v_0$-reduced divisors, superstable configurations.   It is the aim of this paper to present an explicit bijection between the recurrent configurations in an arbitrary hereditary chip-firing model and the spanning trees of a graph.  Our bijection is a modification of the Cori-Le Borgne algorithm \cite{Cori}.

 If we order the elements of $H$ by inclusion, we have a set of maximal elements $A_1, \dots A_k$ which in turn, by the hereditary property, determine $H$.  We note that these maximal elements of $H$ need not be disjoint, i.e. hereditary chip-firing models are not determined by partitions of the vertex set, instead we should think of them as being described by covers of $V(G)\setminus \{v_0\}$.  Moreover, these covers need not be irredundant, instead we ask that the elements of the cover be incomparable.  This allows us to naturally identify hereditary chip-firing models with maximal antichains in the Boolean lattice.  Calculating the number of such maximal antichains is a challenging problem \cite{Duffus}, but this quantity is easily seen to be at least exponential in $n$.


\end{section}

\

\begin{section}{Notation and Terminology}

We take $G$ to be a connected undirected loop less multigraph with vertices labeled $v_0, v_1, \dots v_n$.  Given $X,Y \subset V(G)$, we let $(X,Y) = \{ e \in E(G) : e = (x,y), x\in X, y \in Y\}$, and let $X^c$ denote $V(G) \setminus X$. To describe chip-firing, we begin with a graph $G$ and a configuration $D$ of {\it chips} on $G$.   Formally, a  configuration of chips is a function $D : V(G) \rightarrow \mathbb{Z}^{n+1}$.  For the purposes of this paper we will usually restrict our attention to $D$ such that $D(v_i) \geq 0$ for all $i \neq 0$ and $D(v_0) = -\sum_{i=1}^n D(v_i)$ so that the sum of the values of $D$, called the {\it the degree of $D$}, is 0.  If a vertex $v$ in a configuration of $D$ is seen to have $D(v) <0$, we say that this vertex is {\it in debt}.  The basic operation is {\it firing} whereby a vertex $v$ sends a chip along each of its edges to its neighbors and loses deg$(v)$ chips in the process so that the total number of chips is conserved.  We designate $v_0$ to be the {\it sink vertex} and say that it cannot fire.  This ensures that we cannot continue firing vertices indefinitely.
 The adjacency matrix $A$ of a graph is a $n+1 \times n+1$ matrix with entries $A_{i,j} = \# $ of edges between $v_i$ and $v_j$.  Taking $D$ to be the diagonal matrix with $D_{i,i} = $ degree of $v_i$, the Laplacian of a graph is defined as the difference $D-A$.
 
    For $S \subset V(G)$, we take $\chi_S$ to be the characteristic vector of $S$.  As an abuse of notation we denote $\chi_{\{v_i\}}$ by $\chi_i$.    From a linear algebraic perspective, viewing a configuration $D$ as a vector, if a vertex $v_i$ fires then $D$ is replaced by $D- Q\chi_i$, and more generally if a set $S$ fires we obtain $D-Q\chi_{S}$.  We say that two configurations $D$ and $D'$ are equivalent if there exists some sequence of firings which brings $D$ to $D'$ (possibly including firings by $v_0$ and passing through intermediate  configurations which are negative at vertices other than $v_0$).  Two configurations are seen to be equivalent if their difference is in the integral span of the columns of the Laplacian.  We call a collection of configurations which are equivalent, a chip-firing equivalence class.  

  The ASM (Abelian sandpile model) is defined by placing the additional restriction that vertices may only fire one at a time, whereas in the CFM (cluster firing model), vertices are allowed to fire simultaneously.  We fix a collection $H$ of subsets of $V(G)\setminus \{v_0\}$, those sets which are allowed to fire simultaneously if no vertex is sent into debt,  and call this collection a {\it chip-firing model}.  If each vertex $v_i$ with $i\neq 0$ appears somewhere in $H$, we say that $H$ covers $G$.  If $H$ covers $G$ and $H$ is hereditary, i.e. for every $A \in H$ and $B \subset A$, we have that $B \in H$,  we say that $H$ is a {\it hereditary chip-firing model}. 
  
   Let $H$ be a hereditary chip-firing model on a graph $G$. If a configuration of chips $D$ has no set of vertices $M \in H$ which can fire without some $v \in M$ being sent into debt, we say that $D$ is {\it stable}.  The process of firing sets from $H$ until a configuration becomes stable is called {\it stabilization}.  We say that a set $M\in H$, is {\it ready in $D$}  if this set can fire without sending any vertex into debt, and call a vertex $v$ {\it active} in a configuration $D$ if there exists some $M \subset V(G)\setminus \{v_0\}$ with $v \in M$ which is ready.  Suppose $v \in V(G)$ is active in a configuration $D$.  There may very well be several different maximal ready sets which contain $v$, and these different maximal ready sets might cause $v$ to lose different numbers of chips if they were to fire.  Therefore, we let $m(v,D)$ denote the minimum amount that an active vertex $v$ can lose by firing a maximal ready set in $D$ which contains $v$.
   
   Lemma 1 states that the stabilization of a configuration in a hereditary chip-firing model is well defined, so we denote the stable configuration obtained from $D$ by stabilization as $D^{\circ}$.  A configuration $D$ is said to be {\it reachable} from another configuration $D'$ if there exists a way of adding chips to $D'$ and then firing ready sets to reach $D$.  Because of our convention that the degree of $D$ be zero, we are actually adding configurations of the form $\chi_i-\chi_0$, i.e. subtracting from $v_0$ exactly as many chips as we add to other vertices.  This operation is referred to in \cite{Baker} as the {\it Abel-Jacobi map}.  A configuration $D$ is {\it globally reachable} if it is reachable from every other configuration.  Finally, we call $D$ {\it recurrent}, if it is both stable and globally reachable.  The original motivation for this terminology comes from the observation that if we continue adding chips and stabilizing, the configurations we will see infinitely many times are the recurrent ones.    The recurrent configurations in CFM ($G$-parking functions) are precisely the stable configurations, so there is no need for a discussion of global reachability.   We say that a configuration $\nu$ is {\it critical} if it is stable and $(D-Q\chi_o)^{\circ}=D$.  As with the ASM, a configuration is recurrent if and only if it is critical.  This statement is trivially true for the CFM.  \end{section}

\

\begin{section}{Preliminary Results}
In this section we present the basic results of hereditary chip-firing models.  Hereditary chip-firing models as well as the results of this section were discovered independently of the author by Paoletti \cite{Paoletti1} \cite{Paoletti2}, and Caracciolo, Paoletti and Sportiello\cite{Caracciolo}.  They observe that stabilization in a chip-firing model $H$ is independent of firings if and only if  $H$ is closed under subtraction, i.e. for all $A, B \in H$, we have $A\setminus B \in H$.  They then restrict to the case where for each $v  \in V(G)$, $\{ v\} \in H$.   It is easy to see that a family of subsets of $[n]$ is closed under subtraction and contains all singletons if and only if it is hereditary and covers $[n]$.

\begin{lemma}
  Given a fixed hereditary chip-firing model $H$ on a graph $G$, and a chip-firing configuration $D$ on $G$, the stabilization of $D$ is independent of the firings chosen.  
  
\end{lemma}

\begin{proof}  First, we observe that if $M,N \subset V(G)\setminus \{v_0\}$, $M$ is ready and $N$ fires first, then $M\setminus N$ is ready.  This is because if we fire $N$ and then fire $M\setminus N$, a vertex $v \in M$ loses at most as many chips as if $M$ had fired alone.  More generally, if $M$ is ready and a multi set $N$ fires, i.e. we fire vertices in $N$ a number of times corresponding their multiplicity in N, then $M\setminus N$ is ready.  Let $M_1, \dots, M_s \in H$ and $N_1, \dots, N_t \in H$ correspond to sequences of sets which are fired in two different stabilizations of $D$.  Let $X_{M _q} = \sum_{i=1}^q \chi_{M_i}$ and $X_{B_r} = \sum_{i=1}^r \chi_{B_r} $.  Suppose that $D-QX_{M _s}$ and $D-QX_{N _t}$ are not equal, i.e. the two stabilizations of $D$ are different.  We note that this can occur if and only if $X_{M _s} \neq X_{N _t}$, as $v_0$ does not fire and the kernel of the Laplacian is generated by the all one's vector.  It follows that, without loss of generality, there exists some $l$ maximum such that $X_{M_l}\leq X_{N_t}$ and $X_{M_{l+1}} \nleq X_{N_t}$.  By construction $M_{l+1}$ is ready for  $D- QX_{M_l}$.  Now let $\chi_P = X_{N_{t}}-X_{M_l}$ be the characteristic vector corresponding to the multi set $P$ .  By the first observation, $M_{l+1} \setminus P$ is nonempty and ready for $D- QX_{M_l}-Q\chi_P = D-QX_{N_t} $, but this contradicts the fact that $D-QX_{N_t} $ is stable.

\end{proof}

\begin{theorem}
Given a fixed hereditary chip-firing model $H$ on a graph $G$, there exists a unique recurrent configuration $\nu$ in each chip-firing equivalence class.
\end{theorem}

\begin{proof}

We begin by observing that every chip-firing equivalence class contains at least one recurrent configuration.  In a stable configuration, each vertex $v$ has at most deg$(v)-1$ chips.  Therefore, if we can show that each equivalence class contains a configuration with more than deg$(v)$ chips at each vertex $v$, it would follow that this configuration is globally reachable and hence its stabilization is recurrent.  The technique which we now apply also appears in \cite{Baker}.   Partition the vertices according to their distance from $v_0$.  Let $d$ be the maximum distance of a vertex from $v_0$.  Begin by firing all of the vertices of distance at most $d-1$ from $v_0$.  This has the effect of sending money to the vertices of distance $d$.  Repeat until each such vertex $v$ has at least deg$(v)$ chips.  Now fire all of the vertices of distance at most $d-2$ from $v_0$ until the vertices of distance $d-1$ have at least their degree number of chips.  Working backwards in this way towards $v_0$, we obtain the desired configuration.

We now show that there is at most one recurrent configuration in each equivalence class.  This proof is identical to the argument presented in \cite{Holroyd} -- it is a little surprising that their argument requires no modification.  First, we would like to show that there exists a configuration $\epsilon$ with $\epsilon(v_i) > 0 $ for all $i \neq 0$, such that when we add $\epsilon$ to a recurrent configuration $\nu$ and stabilize, we obtain $\nu$.  Let $D$ be a configuration such that $D(v_i) \geq $ deg$(v_i)$ for all $i \neq 0$.  We will take $\epsilon = D-D^{\circ}$.  Because $\nu$ is recurrent, it is globally reachable, hence there exists some configuration $\zeta$ such that $(D+\zeta)^{\circ} = \nu$.  We are interested in computing $\gamma^{\circ} = (D+\zeta + \epsilon)^{\circ}$.  Because stabilization is independent of firings chosen, we can stabilize $\gamma$ by first stabilizing $D+\zeta$, i.e. $\gamma^{\circ} = ((D+\zeta)^{\circ}+\epsilon)^{\circ} =( \nu+\epsilon)^{\circ}. $ On the other hand, this is also equal to $(D^{\circ}+\zeta +\epsilon)^{\circ} = (D^{\circ}+\zeta +D-D^{\circ})^{\circ} = (\zeta +D)^{\circ} = \nu$.

Assume that there are two different equivalent recurrent configurations $\nu$ and $\nu'$ such that $\nu \sim \nu'$.  By definition, there exists some $f \in \mathbb{Z}^{n+1}$ such that $\nu-\nu' = Qf$, moreover we can take $f$ to be such that $f(v_0)=0$ because the all ones vector is in the kernel of $Q$.  Let $f^+, f^- \in \mathbb{Z}^{n+1}$ be such that $f^+\geq \vec0, f^-\leq \vec0$, and $f^++f^-=f$.  Therefore, there is some configuration $D$ such that $D= \nu-Qf^+= \nu'-Q(-f^-)$.  Note that because $\nu$ and $\nu'$ are stable, it follows that $D$ may have vertices which are in debt.  For any $k \in \mathbb{N}$, $\nu+k\epsilon$ and $\nu'+k\epsilon$ will stabilize to $\nu$ and $\nu'$ respectively, as was shown above.  On the other hand, if we take $k$ to be sufficiently large, we can perform firings defined by $f^+$ and $-f^-$ (by individual vertices for example) to $\nu+k\epsilon$ and $\nu'+k\epsilon$ respectively to obtain the configuration $D+k\epsilon$.  But now we arrive at the contradiction that $D+k\epsilon$ should stabilize to both $\nu$ and $\nu'$.  

\end{proof}

\begin{lemma}  Given a fixed hereditary chip-firing model  on a graph $G$, a chip-firing configuration $\nu$ on $G$ is recurrent if and only if it is critical.  
\end{lemma}

\begin{proof}

Suppose first that $\nu$ is recurrent, but not critical, that is $(\nu-Q\chi_0)^{\circ} = D \neq \nu$.  Then $(\nu+k\epsilon-Q\chi_0)^{\circ}=((\nu+k\epsilon)^{\circ}-Q\chi_0)^{\circ}=(\nu-Q\chi_0)^{\circ}=D$.  Because $\epsilon (v_i) >0$ for all $i \neq 0$, we can take $k$ sufficiently large so that $(\nu+k\epsilon-Q\chi_0)(v_i)> $ deg$(v_i)$ for all $i \neq 0$ and it follows that $D$ is recurrent, a contradiction.  Conversely, suppose that $D$ is not recurrent, but that $D$ is critical, then $(D-kQ\chi_0)^{\circ} = D$ for all $k \in \mathbb{N}$.  If we take $k$ to be sufficiently large, then we can perform firings as in the beginning of Theorem 1 to spread the chips around in the graph and reach a configuration which has at least degree number of chips at each vertex.  It follows that $D$ is globally reachable, hence recurrent, a contradiction.

\end{proof}

\begin{lemma}
The number of chip-firing equivalence classes on a graph $G$ is the same as the number of spanning trees of $G$.
\end{lemma}

\begin{proof}
Let $\bar{Q}$ denote the matrix obtained from $Q$ by deleting the first row and column, i.e. the row and column corresponding to $v_0$.  This matrix, called the reduced Laplacian of a graph is known to have full rank as $G$ is connected, and by the matrix tree theorem $det (\bar{Q})$ is equal to the number of spanning trees of $G$ \cite{Kirchoff}.  By ignoring the values of $v_0$ in our configuration, we see that the number of different chip-firing equivalence classes is the number of cosets for the image of $\bar{Q}$ and this index is given by $det(\bar{Q})$.
\end{proof}

\end{section}

\

\begin{section}{Bijection}

This algorithm is a modification of the Cori-Le Borgne algorithm \cite{Cori} as presented in \cite{Baker2}.  Their algorithm can be viewed as a variant of Dhar's burning algorithm \cite{Dhar2}.  We will need to call Dhar's burning algorithm as a subroutine, so we first begin by describing this method, and do so in the context of the cluster firing model where the author believes it is more naturally understood.  One might argue that the brilliance of Dhar's algorithm is that its discovery occurred in the context of the Abelian sandpile model, where its application is less obvious.  

Given a recurrent configuration $\nu$ for the sandpile model $K^+ - \nu ={\bar \nu}$ is a recurrent configuration in the cluster firing model, where $K^+(v)= deg(v)-1$ for all $V(G)\setminus\{v_0\}$.  The interested reader can prove this fact for themselves using Lemma 3 or look to \cite{Baker} for an alternate proof.  This allows a bijection for one model to be ``dualized" to produce a bijection for the other model.  The bijection presented here is the first bijection which the author is aware of that applies directly to both models without exploiting this duality.

As was mentioned in the introduction, the recurrent configurations in the cluster firing model are precisely the stable configurations, therefore, to check that a configuration $\nu$ is recurrent, we need only check that there exists no set $A \subset V(G)\setminus{v_0}$ which can fire without sending a vertex into debt.  {\it A priori} we would need to check an exponential number of sets to be sure that $\nu$ was reduced, but Dhar's observation is that it's sufficient to check only $n$ such sets.  Assume that $\nu$ is stable and begin by firing $A_1=V(G)\setminus{v_0}$.  By assumption, there exists at least one vertex $v$ which is sent into debt.  Remove $v$ from $A_1$ and continue firing sets in $\nu$ and removing vertices sent into debt until reaching the empty set.   

Here is why this works:  suppose that $B \in V(G)\setminus{v_0}$ is ready in $\nu$, but that we have a collection $A_1, \dots, A_n$ of sets which were obtained from a run of Dhar's algorithm.  There exists $i$ maximum such that $B\subset A_i $.  It follows that $A_{i-1} = A_i \setminus v$, with $v \in B$, where $v$ was sent into debt by $A_i$, but if we fire $A_i \setminus B$, $v$ may only gain chips, and $v$ is supposedly able to fire in $B$ without being sent into debt.  Firing $A_i \setminus B$ and then $B$ is the same as firing $A_i$, contradicting the fact that $v$ was sent into debt by $A_i$.

Dhar's burning algorithm earns its name from the following alternate description:  Place $D(v)$ firefighters at each vertex, and start a fire at $v_0$.  The fire spreads through the graph along the edges, but is prevented from passing through vertices by the firefighters located there.  When the number of edges burned incident to a vertex is greater than the number of firefighters present, the firefighters are overpowered and the fire burns through the vertex.  A configuration is stable in the cluster firing model if and only if the fire burns through the entire graph.  Dhar noticed that by burning in a systematic way, this algorithm can be turned into a bijection between the recurrent configurations and the spanning trees.  

In the Cori-Le Borgne algorithm, the edges are burned in a different order to produce an ``activity preserving" bijection.  To describe the Cori-Le Borgne algorithm, we begin with an arbitrary ordering of the edges $e_1, e_2, \dots, e_m \in E(G)$.  The setup is the same as with Dhar, except that we burn one edge at a time, always taking the edge with the smallest label connecting the burnt vertices to the non burnt vertices.  When an edge picked/burned causes the firefighters at a vertex to be overpowered and the vertex to be burnt, we mark this edge.  It is clear that if the fire burns through the graph, these marked edges form a spanning tree .  Converesly, if we start with a tree and begin burning the edges of our graph one at a time, the edges of the tree tell us when we should burn a vertex, hence how many firefighters (chips) a vertex should have.  This shows that the algorithm produces a bijection between the recurrent configurations and spanning trees.
\

We now give a third characterization of recurrent configurations.  This definition is the the one which will be used in our bijection.

\begin{lemma}
A configuration $\nu$ is critical if and only if any maximal sequence of firings by active vertices brings $\nu-Q\chi_0$ back to $\nu$.
\end{lemma}

\begin{proof}
Here, we are allowing active vertices to fire even though this may cause them to go into debt.  If a configuration $\nu$ is critical, it is clear that we can continue firing active vertices in the ready sets and eventually return to $\nu$.  Conversely, suppose that there exists some firing of individual active vertices which brings $\nu-Q\chi_0$ back to $\nu$, but that $\nu$ is not critical.  If this is the case, there must be some vertex $v \in V(G) \setminus \{v_0\}$ which was never fired in the stabilization of $\nu-Q\chi_0$.  We might take $v$ to be the first such vertex, but observe that this situation may only occur if a vertex of the same type has already been fired causing $v$ to become active, a contradiction.
\end{proof}
\

We let $\mathcal N$ denote the set of recurrent configurations from a hereditary chip-firing model on a graph $G$, and $\mathcal{T}$ denote the set of spanning trees of $G$.  Recall $m(v,D)$ is minimum amount that $v$ can lose by firing a ready set in $D$ which contains $v$.  We now explain our bijection between recurrent configurations in a fixed hereditary chip-firing model $H$ on a graph $G$ and the spanning trees of $G$. Here is the algorithm $\sigma$ for taking a recurrent configuration $D$ and producing a spanning tree $T = \sigma(D)$:

\

\

$\sigma : \mathcal{N} \rightarrow \mathcal{T}$

\begin{itemize}

\item Begin with $X= v_0$, $R= \emptyset$, $T=\emptyset$. 

\item Take $e_i \in (X,X^c)$ with $i$ minimum such that $e_i \notin R$.

\noindent Let $e_i = (u,v)$ with $u \in X$ and $v \in X^c$.  

\item If $D(v) < m(v,D-Q\chi_{X})- |\{e \in R : e=(w,v)\} \cup \{e_i\}|$, set $R:= R\cup \{e_i\}$.  

\item If $D(v) = m(v,D-Q\chi_{X})- |\{e \in R : e=(w,v)\} \cup \{e_i\}|$, we set $T:= T\cup\{e_i\}$, $R:=\emptyset$, and $X:= X \cup \{v\}$.

\end{itemize}

\

We now describe our algorithm $\gamma$ for taking a tree $T$ and producing a recurrent configuration, $\gamma(T)$.  This process has two parts.  First we use $T$ to construct a total order on the vertices:  

\

\

$\gamma : \mathcal{T} \rightarrow \mathcal{N}$ 

 Part 1:

\begin{itemize}

\item Begin with  $X_0= \{v_0\}=\{w_0\}$, and $R= \emptyset$.  

\item Take $e_i \in (X,X^c)$ with $e_i \notin R$ and $i$ minimum.  

\item If $e_i \notin T$, set $R:= R \cup \{e_i\}$

\item If  $e_i \in T$ with $e_i = (u,v)$ and $v \in X^c$, set $v= w_i$, $X_i = X_{i-1}\cup\{w_i\}$, and $R:=\emptyset$.

\end{itemize}

\

This process terminates when we have observed every edge of $T$, at which point we have acquired a total order $w_0 < w_1 < \dots < w_n$ on the vertices.  

\

The idea is to now run the algorithm``backwards" using the total order to reconstruct $D$:

\

$\gamma : \mathcal{T} \rightarrow \mathcal{N}$ 

Part2:

\begin{itemize}

\item Begin with $Y_{0}=\{w_0, \dots, w_{n-1}\}$ and $R= \emptyset$.

\item Assume that at the $i$th iteration of step 1, the values of $D(w_j)$ for $j\geq n-i+1$ have already been calculated.  

\item  {\bf Step 1}:  Fire $Y_{i} = \{w_0, \dots, w_{n-1-i}\}$, i.e. compute $D-Q\chi_{ Y_i}$.

\item  Find the maximal ready sets in $V(G) \setminus Y_i$ for $D-Q\chi_{ Y_i}$ by applying Dhar's burning algorithm to $A_j \cap (V(G) \setminus Y_i)$ for each maximal element $A_j \in H$.

\item Compute $m(w_{n-i}, D- Q\chi_{Y_{i}})$.

\item {\bf Step 2}: Pick $e_k $ with $k$ minimum so that $e_k \notin R$.  Let $e_k = (u, v)\in (Y_i, Y_i^c)$.

\item  If $e_k \notin T$, set $R:= R \cup \{e_k\}$, and go to step 2.

\item If $e_k \in T$, we know that $e_k = (u, w_{n-i})$. 

\noindent Set $D(w_{n-i})= m(w_{n-i}, D- Q\chi_{Y_{i}})-|\{ e \in R: e = (v, w_{n-i})\}\cup \{e_k\}|$, $R:=\emptyset$, and $Y_i := Y_{i+1} $, then go to step 1.

\end{itemize}

\

\begin{theorem}
The operations, $\sigma$ and $\gamma$ are inverse to each other and induce a bijection between the recurrent configurations of a hereditary chip-firing model $H$ on a graph $G$ and the spanning trees of $G$.
\end{theorem}

\begin{proof}
First we claim that $\gamma \circ \sigma$ is the identity map on the recurrent configurations.  Let $D$ be recurrent and $\sigma(D) =T$ a spanning tree.  Observe that the total order produced on the vertices of $G$ during the run of $\gamma$ on $T$ is the same as the order in which the vertices are processed during $\sigma$ run on $D$.  Given this total order on the vertices, the algorithm $\gamma$ is designed so as to produce the configuration $D$ such that $\sigma(D)=T$.  It follows that $\sigma$ is injective, and by Lemma 3, $\sigma$ is an injective map between two sets with the same cardinality.  It follows that $\sigma$ is a bijection with explicit inverse $\gamma$.
\end{proof}

\

\end{section}

\bibliographystyle{model1a-num-names}
\bibliography{<your-bib-database>}

\begin{thebibliography}{00}



\bibitem{Backman}
 A. Asadi and S. Backman,  Chip-firing and Riemann-Roch theory for directed graphs (2011) http://arxiv.org/abs/1012.0287


\bibitem{Bak}
P. Bak, C. Tang and K. Wiesenfeld (1987). Self-organized criticality: an explanation of 1/f noise, Physical Review Letters 59 (4): 381Ð384.

\bibitem{Baker}
M. Baker and S. Norine, Riemann-Roch and Abel-Jacobi theory on a finite graph, Adv. Math. 215 (2007), no. 2, 766Ð788.

\bibitem{Baker2}
M.Baker and F. Shokrieh, Chip Firing Games, Potential Theory on Graphs, and Spanning Trees, to appear in J. Combinatorial Theory Series A.

\bibitem{Benson}
B. A. Benson, D. Chakrabarty, and P. Tetali. G-parking functions, acyclic orientations and spanning trees. Discrete Mathematics, 310(8):1340Ð1353, 2010

\bibitem{Bernardi}
O. Bernardi, Tutte polynomial, subgraphs, orientations and sandpile model: new connections via embeddings.
. Electronic Journal of Combinatorics, Vol 15(1), R109 (2008)

\bibitem{Biggs}
N. Biggs, Algebraic potential theory on graphs. Bull. London Math. Soc. 29 (1997) 641-682.

\bibitem{Biggs2}
N. Biggs and P. Winkler, Chip-firing and the chromatic polynomial. CDAM Research Report Series, LSE-CDAM-97-03, February 1997.

\bibitem{Bjorner}
A. Bj\"orner, L. Lov\'asz,  and P.Shor: Chip-firing games on graphs, Europ. J. Comb. 12 (1991), 283-291.

\bibitem{Bond}
B. Bond and L. Levine, Abelian Networks1: Foundations and Examples (2012) , www.math.cornell.edu/~levine/abelian-networks-I.pdf

\bibitem{Caracciolo}
S. Caracciolo, G. Paoletti and A. Sportiello: Ò Multiple and inverse top- plings in the Abelian Sandpile Model Ó preprint (2011) [arXiv:1112.3491]

\bibitem{Chebkin}
 D. Chebikin and P. Pylyavskyy. A family of bijections between G-parking functions and spanning
trees. J. Combin. Theory Ser. A, 110(1):31Ð41, 2005.

\bibitem{Chung}
F. Chung and R. B. Ellis. A chip-firing game and Dirichlet eigenvalues. Discrete Math., 257(2-
3):341Ð355, 2002. Kleitman and combinatorics: a celebration (Cambridge, MA, 1999).

\bibitem{Cori}
R. Cori and Y. Le Borgne. The sand-pile model and Tutte polynomials. Adv. in Appl. Math.,
30(1-2):44Ð52, 2003. Formal power series and algebraic combinatorics (Scottsdale, AZ, 2001)

\bibitem{Dhar}
D. Dhar, Self-organised critical state of the sandpile automaton models. Physical Review Lettters
64 (1990), 1613-1616.

\bibitem{Dhar2}
D. Dhar and S. N. Majumdar, Equivalence of the abelian sandpile model and the $q \rightarrow 0$ limit of the Potts model, Physica A185 (1992) 129.

\bibitem{Duffus}
D. Duffus, B. Sands and P. Winkler,  Maximal chains and antichains in Boolean lattices, SIAM J. Disc. Math., 3:2 (1990), pp. 197-205.

\bibitem{Gabrielov}
A. Gabrielov. Abelian avalanches and Tutte polynomials. Phys. A, 195(1-2):253Ð274, 1993

\bibitem{Goles}
E. Goles and E. Prisner, Source reversal and chip firing on graphs, Theoret. Comput. Sci. 233 (2000) 287-295 

\bibitem{Holroyd}
A. E. Holroyd, L. Levine, K. Meszaros, Y. Peres, J. Propp and D. B. Wilson, Chip-firing and rotor-routing on directed graphs,  In and Out of Equilibrium II, Progress in Probability vol. 60 (Birkhauser, 2008)

\bibitem{Kirchoff}
 G. Kirchhoff. \"{U}ber die Aufl\"{o}sung der Gleichungen, auf welche man bei der Untersuchung der
linearen Verteilung galvanischer Str\"{o}me gef\"{u}hrt wird. Ann. Phys. Chem., (72):497Ð508, 1847

\bibitem{Knuth}
D.E. Knuth, Sorting and Searching, The Art of Computer Programming, vol. 3, Addison- Wesley, Reading, MA, 1973.

\bibitem{Kostic}
D. N. Kosti\'c, Bijections Between Multiparking Functions, Dirichlet Configurations, and Descending R -Traversals, Journal: Annals of Combinatorics - ANN COMB , vol. 13, no. 1, pp. 103-114, 2009

\bibitem{Kostic2}
D. N. Kosti\'c, C. Yan,  Multiparking Functions, Graph Search, and Tutte Polynomial., Advances in Applied Mathematics, Vol. 40 (2008), 73-97.

\bibitem{Manjuneth}
M. Manjunath and B. Sturmfels, Monomials, binomials, and Riemann-Roch , 18 pp., to appear in Journal of Algebraic Combinatorics. 

\bibitem{Postnikov}
A. Postnikov and B. Shapiro. Trees, parking functions, syzygies, and deformations of monomial
ideals. Trans. Amer. Math. Soc., 356(8):3109Ð3142 (electronic), 2004.


\bibitem{Paoletti1}
Guglielmo Paoletti, July 11 2007: Master in Physics at University of Milan,  defending thesis ``Abelian sandpile models and sampling of trees and forests" ; supervisor: Prof. S. Caracciolo.

\bibitem{Paoletti2}
Guglielmo Paoletti, February 2 2012 : PhD in Physics at Graduate School of Basic Sciences ÒGalileo GalileiÓ - Physics, University of Pisa, defending the thesis ÒDeterministic Abelian Sandpile Models and PatternsÓ; supervisor: Prof. S. Caracciolo.

\bibitem{Pak}
I. Pak, Tree Bijections, http://www.math.ucla.edu/~pak/lectures/bij-trees-talk1.pdf, IPAM workshop, September 9, 2009

\bibitem{Pittman}
 J. Pitman and R. Stanley, A Polytope Related to Empirical Distributions, Plane Trees, Parking Functions, and the Associahedron, Discrete Comput. Geom., 27 (2002), 603-634.

\bibitem{Pittman}
R. Stanley, Enumerative Combinatorics, volume 1, Wadsworth and Brooks/Cole, Pacific Grove, CA 1986; second printing, Cambridge University Press, Cambridge/New York, 1996.

\bibitem{Pittman2}
R. Stanley, Enumerative Combinatorics, volume 2, Cambridge University Press, 1999.

\bibitem{Spencer}
J. Spencer (1986), Balancing vectors in the max norm, Combinatorica 6, 55-66.

\bibitem{Tardos}
G. Tardos, Polynomial bound for a chip firing game on graphs, SIAM Journal on Discrete Mathematics 1 (1988), 397-398.

\bibitem{Yan}
 C. H. Yan, Generalized Parking Functions, Tree Inversions and Multicolored Graphs, Advances in Applied Mathematics, Vol. 27, 641--670, 2001.



\end{thebibliography}



\

 {\bf School of Mathematics, Georgia Institute of Technology, Atlanta, Georgia 30332-0160, USA}

 email:  sbackman@math.gatech.edu

\end{document}